\documentclass{amsart}
\RequirePackage[utf8]{inputenc}
\usepackage{amsmath,amsthm,amsfonts,amssymb,amscd,amsbsy,hyperref}
\usepackage[all]{xy}
\usepackage{graphicx,pgf,caption,subcaption}
\usepackage{enumerate}
\usepackage{dsfont}

\newcommand{\dd}{\mathrm{d}}
\newcommand{\id}{\operatorname{Id}}

\newcommand{\KK}{\mathds K}
\newcommand{\R}{\mathds R}
\newcommand{\C}{\mathds C}
\newcommand{\Hr}{\mathds H}
\newcommand{\Ca}{\mathds{C}\mathrm{a}}
\newcommand{\SO}{\mathsf{SO}}
\renewcommand{\O}{\mathsf O}

\newcommand{\U}{\mathsf{U}}
\newcommand{\Sp}{\mathsf{Sp}}
\newcommand{\Spin}{\mathsf{Spin}}
\newcommand{\G}{\mathsf{G}}
\newcommand{\K}{\mathsf{K}}
\renewcommand{\H}{\mathsf{H}}
\newcommand{\N}{\mathsf{N}}
\newcommand{\A}{\mathsf{A}}
\newcommand{\Sym}{\operatorname{S}(\Lambda^2 V)}
\renewcommand{\b}{\mathfrak b}
\newcommand{\g}{\mathrm g}

\newcommand{\Gr}{\operatorname{Gr}_2}
\newcommand{\mm}{\mathfrak m}
\newcommand{\h}{\mathfrak h}
\newcommand{\p}{\mathfrak p}
\renewcommand{\k}{\mathfrak k}
\renewcommand{\a}{\mathfrak a}
\newcommand{\diag}{\operatorname{diag}}

\renewcommand{\Sym}{\operatorname{Sym}}

\newcommand{\Ad}{\operatorname{Ad}}

\allowdisplaybreaks

\hypersetup{
    pdftoolbar=true,
    pdfmenubar=true,
    pdffitwindow=false,
    pdfstartview={FitH},
    pdftitle={Strongly nonnegative curvature},
    pdfauthor={Renato G. Bettiol and Ricardo A. E. Mendes},
    pdfsubject={},
    pdfkeywords={}
    pdfnewwindow=true,
    colorlinks=true, 
    linkcolor=blue,
    citecolor=blue,
    urlcolor=black,
}

\newtheorem{theorem}{Theorem}[]
\newtheorem{lemma}[theorem]{Lemma}

\newtheorem{corollary}[theorem]{Corollary}

\newtheorem{mainthm}{\sc Theorem}

\theoremstyle{definition}

\theoremstyle{remark}
\newtheorem{remark}[theorem]{Remark}

\title{Strongly nonnegative curvature}

\author[R. G. Bettiol]{Renato G. Bettiol}
\author[R. A. E. Mendes]{Ricardo A. E. Mendes}

\address{\begin{tabular}{lll}
University of Pennsylvania & & Universit\"at M\"unster \\
Department of Mathematics & & Mathematisches Institut \\
209 South 33rd St  & & Einsteinstr.\ 62 \\
Philadelphia, PA, 19104-6395, USA & & D-48149 M\"unster, Germany\\
\emph{E-mail address}: {\tt rbettiol@math.upenn.edu} & & \emph{E-mail address}: {\tt mendes@uni-muenster.de}\\[0.3cm]
{\it Current address for R.\ G.\ Bettiol:} && \\
Max Planck Institute for Mathematics && \\
Vivatsgasse 7\\ 53111 Bonn, Germany &&
\end{tabular}
}

\dedicatory{Dedicated to Karsten Grove on his 70th birthday}

\allowdisplaybreaks
\numberwithin{equation}{section}
\numberwithin{theorem}{section}

\thanks{The first named author was partially supported by the NSF grant DMS-1209387, USA. The second named author is supported by SFB878 Groups, Geometry \& Actions.}
\subjclass[2010]{53B20, 53C20, 53C21, 53C30, 53C35} 

\date{\today}

\begin{document}
\begin{abstract}
We prove that all currently known examples of manifolds with nonnegative sectional curvature satisfy a stronger condition: their curvature operator can be modified with a $4$-form to become positive-semidefinite.
\end{abstract}

\maketitle

\section{Introduction}

The geometry and topology of manifolds with nonnegative and positive sectional curvature ($\sec\geq0$ and $\sec>0$) have been of great interest since the early days of global Riemannian geometry, and remain exciting research areas with many challenging problems. Despite the natural ties between the classes of manifolds with $\sec\geq0$ and $\sec>0$, there is a sharp contrast in the number of constructions and examples, see Wilking~\cite{wilking-survey} and Ziller~\cite{bible} for surveys. On the one hand, the only currently known examples of closed manifolds with $\sec>0$  different from compact rank one symmetric spaces (CROSS) occur in dimensions $6$, $7$, $12$, $13$ and $24$. On the other hand, a wealth of examples of closed manifolds with $\sec\geq0$ have been produced (beyond homogeneous spaces and biquotients), notably by methods developed by Cheeger~\cite{cheeger} and Grove and Ziller~\cite{grove-ziller-annals,grove-ziller-tams}.
It follows from our previous work~\cite{strongpos,moduli-flags} that almost all known examples of closed manifolds with $\sec>0$ actually satisfy a stronger curvature condition, called \emph{strongly positive curvature}. The purpose of this paper is to show that \emph{all known examples} of manifolds with $\sec\geq0$ have \emph{strongly nonnegative curvature}. This further corroborates the importance of strongly nonnegative and positive curvature in the study of $\sec\geq0$ and $\sec>0$.

A Riemannian manifold $(M,\g)$ is said to have strongly nonnegative curvature if, for all $p\in M$, there exists a $4$-form $\omega\in\wedge^4 T_pM$ such that the modified curvature operator $(R+\omega)\colon \wedge^2 T_pM\to\wedge^2 T_pM$ is positive-semidefinite. This is an intermediate condition between $\sec\geq0$ and positive-semidefiniteness of the curvature operator. It is worth recalling that manifolds satisfying the latter have been classified~\cite{bw,wilking-survey}, see Section~\ref{sec:basics} for details. Some of the key properties of strongly nonnegative curvature is that it is preserved by products, Riemannian submersions and Cheeger deformations \cite[Thm.\ A, Thm.\ B]{strongpos}, see also \cite[\S 6.4]{strongpos}. In particular, since any compact Lie group $\G$ with bi-invariant metric has positive-semidefinite curvature operator, all compact homogeneous spaces $\G/\H$ and all compact biquotients $\G/\!\!/\H$ have metrics with strongly nonnegative curvature.

Using a gluing method inspired by the construction of Berger spheres, Cheeger~\cite{cheeger} produced 
another class of closed manifolds with $\sec\geq0$. Our first main result is that all manifolds in this class also have strongly nonnegative curvature:

\begin{mainthm}\label{thm:A}
The connected sum of any two compact rank one symmetric spaces (with any orientation) admits a metric with strongly nonnegative curvature.
\end{mainthm}

We remark that some manifolds in Theorem \ref{thm:A} are diffeomorphic to biquotients, while others are not even homotopy equivalent to biquotients \cite{totaro-biquotients}, see Remark~\ref{rem:totaro}.

A significant generalization of the gluing construction in \cite{cheeger} was achieved by Grove and Ziller~\cite{grove-ziller-annals}, in the context of \emph{cohomogeneity one manifolds}. These are manifolds with an isometric group action whose orbit space is $1$-dimensional, see Section~\ref{sec:cohom1} for details. Our second main result is that their method to produce metrics with $\sec\geq0$ actually yields strongly nonnegative curvature:

\begin{mainthm}\label{thm:B}
Every cohomogeneity one manifold whose nonprincipal orbits have codimension $\leq2$ admits an invariant metric with strongly nonnegative curvature.
\end{mainthm}

The class of manifolds in Theorem \ref{thm:B} is surprisingly rich. For instance, it includes all $4$ oriented diffeomorphism types homotopy equivalent to $\R P^5$, see \cite[Thm.~G]{grove-ziller-annals}. Even more interestingly, it includes a number of total spaces of principal $\G$-bundles, which can be used to construct metrics with strongly nonnegative curvature on associated vector bundles and sphere bundles (see Corollary~\ref{cor:C}). Remarkably, in combination with other techniques, this implies that \emph{all exotic $7$-spheres} admit metrics with strongly nonnegative curvature (see Section~\ref{sec:cohom1} for details).

Constructions of metrics with nonnegative sectional curvature on vector bundles can be interpreted as instances the ``converse'' to the Soul Theorem of Cheeger and Gromoll~\cite{cheeger-gromoll}. 
This celebrated result states that any complete open manifold $M$ with $\sec\geq0$ has a totally convex compact submanifold $S\subset M$ without boundary, called the \emph{soul} of $M$, such that $M$ is diffeomorphic to the normal bundle of $S$ in $M$. Observe that if $M$ has strongly nonnegative curvature, then so does its soul $S$, as it is a totally geodesic submanifold~\cite[Prop.\ 2.6]{strongpos}. The ``converse" question of which vector bundles over closed manifolds with $\sec\geq0$ admit a complete metric with $\sec\geq0$ has been studied by several authors.
It follows from our results that all the progress made to date regarding this problem can be transplanted to the context of strongly nonnegative curvature (see Corollary~\ref{cor:C}).

In the context of complete open manifolds with $\sec\geq0$, Guijarro~\cite{guijarro} 
proved the existence of an ``improved" metric which is isometric to a product outside a neighborhood of the soul. Our third main result is that the same improvement can be obtained with strongly nonnegative curvature:

\begin{mainthm}\label{thm:D}
Let $(M,\g)$ be a complete open manifold with strongly nonnegative curvature and soul $S$. There exists another metric $\g'$ on $M$ with strongly nonnegative curvature, such that $S$ remains a soul, and $(M,\g')$ is isometric to a product $\nu_1(S)\times [1,+\infty)$ outside a compact neighborhood of $S$.
\end{mainthm}

As a consequence of Theorem~\ref{thm:D}, it is possible to construct metrics with strongly nonnegative curvature on \emph{doubles} of any open manifolds with strongly nonnegative curvature. Recall that the double of an open manifold (or manifold with boundary) $M$ is the closed manifold obtained by gluing two copies of $M$ together along their boundary. For instance, one can build closed examples of manifolds with strongly nonnegative curvature taking the double of any of the vector bundles in Corollary~\ref{cor:C}.

The constructions in Theorems~\ref{thm:A}, \ref{thm:B}, and \ref{thm:D} comprise an exhaustive list of all the currently known methods to produce manifolds with $\sec\geq0$. Therefore, as claimed in the first paragraph, \emph{all known examples} of manifolds with $\sec\geq0$ have strongly nonnegative curvature.

Besides the fundamental fact that strongly nonnegative curvature is preserved under Riemannian submersions~\cite{strongpos}, there are two main technical tools needed to prove the above results. The first (Lemma~\ref{lem:fourthirds}) is that bi-invariant metrics on Lie groups retain strongly nonnegative curvature after being dilated by a factor of up to $\tfrac43$ in the direction of an abelian subalgebra. This result is a strengthening of a result in Grove and Ziller~\cite[Prop.\ 2.4]{grove-ziller-annals}, see also Ziller~\cite[Lemma 2.9]{bible}, using the same key fact that such dilations are ``backwards" Cheeger deformations, or submersions from a certain semi-Riemannian manifold. The second technical result (Lemma~\ref{lem:cheeger}) asserts that certain disk bundles whose boundary is a homogeneous space with strongly positive curvature have a metric with strongly nonnegative curvature which is a product near the boundary. This is proved through certain estimates that generalize those in Cheeger~\cite{cheeger}.

This paper is organized as follows. Section~\ref{sec:basics} provides a recollection of the definitions and basic properties of strongly nonnegative curvature, as well as a discussion of basic examples. Constructions of metrics with strongly nonnegative curvature on cohomogeneity one manifolds are given in Section~\ref{sec:cohom1}, where Theorem~\ref{thm:B} is proved and its consequences for associated bundles are described. In Section~\ref{sec:cheegermflds}, we explain a method to endow certain disk bundles with strongly nonnegative curvature, leading to the proof of Theorem~\ref{thm:A}. Finally, Section~\ref{sec:guijarro} contains the proof of Theorem~\ref{thm:D}.

\medskip
\noindent
{\bf Acknowledgements.}
It is a pleasure to thank Burkhard Wilking and Wolfgang Ziller for suggestions regarding Lemmas \ref{lem:fourthirds} and  \ref{lem:cheeger} respectively.

\section{Definitions and Basic Properties}\label{sec:basics}

A detailed account on strongly positive and nonnegative curvature can be found in \cite{mybook,thesis,strongpos,moduli-flags}. As a service to the reader, a short summary is provided below.

\subsection{Modified curvature operators}

Let $(M,\g)$ be a Riemannian manifold. Using the inner products induced by $\g$, identify all exterior powers $\wedge^k T_pM$ with their duals $\wedge^k T_pM^*$.
Denote by $\Sym^2(\wedge^2 T_pM)$ the space of symmetric linear operators $S\colon \wedge^2 T_pM\to\wedge^2 T_pM$, and by $\mathfrak b\colon\Sym^2(\wedge^2 T_pM)\to\wedge^4 T_pM$ the \emph{Bianchi map}
\begin{equation*}
\mathfrak b(S)(X,Y,Z,W)=\tfrac13\Big(\langle S(X\wedge Y),Z\wedge W\rangle + \langle S(Y\wedge Z),X\wedge W\rangle + \langle S(Z\wedge X),Y\wedge W\rangle\Big).
\end{equation*}
Furthermore, identify $\wedge^4 T_pM$ as a subspace of $\Sym^2(\wedge^2 T_pM)$, by means of
\begin{equation}\label{eq:wedge4}
\langle \omega(X\wedge Y),Z\wedge W\rangle=\omega(X,Y,Z,W),
\end{equation}
so that $\Sym^2(T_pM)=\ker\mathfrak b\oplus\wedge^4 T_pM$ is an orthogonal direct sum decomposition, and $\mathfrak b$ is the orthogonal projection operator onto $\wedge^4 T_pM$.

With the above setup, we may add to the curvature operator $R\in\ker\mathfrak b$ of $(M,\g)$ any $4$-form $\omega\in\wedge^4 T_pM$, and the resulting \emph{modified curvature operator} $(R+\omega)\in\Sym^2(\wedge^2 T_pM)$ has the same sectional curvature function as $R$.
Indeed, by \eqref{eq:wedge4}, the quadratic form associated to $\omega\in\wedge^4 T_pM$ vanishes on the Grassmannian of (oriented) $2$-planes $\Gr(T_pM)=\{X\wedge Y\in\wedge^2 T_pM:\|X\wedge Y\|^2 =1\}$, and hence
\begin{equation}\label{eq:samesec}
\sec(X\wedge Y)=\langle R(X\wedge Y),X\wedge Y\rangle=\langle (R+\omega)(X\wedge Y),X\wedge Y\rangle.
\end{equation}

\subsection{Strongly nonnegative curvature}
The manifold $(M,\g)$ is said to have \emph{strongly nonnegative curvature} if, for all $p\in M$, there exists $\omega\in\wedge^4 T_pM$ such that the modified curvature operator $R+\omega$ is positive-semidefinite.

Strongly nonnegative curvature is clearly an intermediate curvature condition between $\sec\geq0$ and positive-semidefiniteness of the curvature operator. All these curvature conditions are equivalent in dimensions $\leq3$, and strongly nonnegative curvature remains equivalent to $\sec\geq0$ in dimension $4$, see \cite{Thorpe72} and \cite[Prop.\ 6.83]{mybook}.

\subsection{Basic properties}
Elementary arguments show that products and totally geodesic submanifolds of manifolds with strongly nonnegative curvature also have strongly nonnegative curvature \cite[\S 2]{strongpos}. In addition, strongly nonnegative curvature is preserved under Riemannian submersions. This fundamental result was established in \cite{strongpos}, by rewriting the Gray-O'Neill formula \cite[Thm.\ 9.28f]{besse} that relates curvature operators of a Riemannian submersion $\pi\colon\overline M\to M$ and its $A$-tensor as
\begin{equation}\label{eq:oneill}
\begin{aligned}
\langle R(X\wedge Y),Z\wedge W\rangle &=\langle \overline R(\overline X\wedge\overline Y),\overline Z\wedge\overline W\rangle + 3\langle\alpha(\overline X\wedge\overline Y),\overline Z\wedge\overline W\rangle \\
&\quad -3\mathfrak b(\alpha)(\overline X,\overline Y,\overline Z,\overline W),
\end{aligned}
\end{equation}
where $\alpha\in\Sym^2(\wedge^2 T_pM)$ is the positive-semidefinite operator $\alpha=A^*A$, i.e.,
\begin{equation*}
\langle\alpha(X\wedge Y),Z\wedge W\rangle=\langle A_X Y,A_Z W\rangle.
\end{equation*}
According to \eqref{eq:oneill}, if there exists $\overline\omega\in\wedge^4 T_p\overline M$ such that $\overline R+\overline\omega$ is positive-semidefinite, then $\omega=\big(\overline\omega +3\mathfrak b(\alpha)\big)|_{\wedge^2 T_pM}$ is such that $R+\omega$ is positive-semidefinite.
Similar arguments also show that strongly nonnegative curvature is preserved under Cheeger deformations \cite[\S 2.5]{strongpos}.

\subsection{Basic examples}
The simplest examples of manifolds with strongly nonnegative curvature are
those whose curvature operator is positive-semidefinite. Closed manifolds with this property have been classified, mainly through the work of B\"ohm and Wilking~\cite{bw}, see Wilking~\cite[Thm.\ 1.13]{wilking-survey}.
Namely, each factor in the de~Rham decomposition of the universal covering of such a manifold is isometric to one of:
\begin{enumerate}[\rm (i)]
\item Euclidean space;
\item Sphere with positive-semidefinite curvature operator;
\item Compact irreducible symmetric space;
\item Compact K\"ahler manifold biholomorphic to $\C P^n$ whose restriction of the curvature operator to real $(1,1)$-forms is positive-semidefinite.
\end{enumerate}
An important subfamily are Lie groups $\G$ with a bi-invariant metric $Q$. Recall that the curvature operator $R_\G \colon \wedge^2\!\mathfrak g\to\wedge^2\mathfrak g$ of $(\G,Q)$ is given by
\begin{equation}\label{eq:RG}
\langle R_\G(X\wedge Y),Z\wedge W\rangle=\tfrac14 Q([X,Y],[Z,W]),
\end{equation}
which is clearly positive-semidefinite.

Since Riemannian submersions preserve strongly nonnegative curvature, all compact homogeneous spaces $\G/\H$ and all compact biquotients $\G/\!\!/\H$ have metrics with strongly nonnegative curvature. 
For instance, one may take on $\G/\H$ the so-called \emph{normal homogeneous metric}, that is, the metric induced by the bi-invariant metric $Q$ on $\G$ via the quotient map, and similarly for $\G/\!\!/\H$.

\begin{remark}
It is an interesting question whether the moduli spaces of homogeneous metrics with strongly nonnegative curvature and $\sec\geq0$ coincide on a given compact homogeneous space.
This has been studied for Wallach flag manifolds in \cite{moduli-flags} and Berger spheres in \cite{thesis,strongpos}. In the former, these moduli spaces coincide, but that is not the case in the latter. In fact, the spheres $S^{4n+3}=\Sp(n+1)/\Sp(n)$ and $S^{15}=\Spin(9)/\Spin(7)$ endowed with the Berger metric $\g_\lambda=\lambda\,\g_\mathcal V\oplus\g_\mathcal H$ have $\sec\geq0$ for all $0<\lambda\leq\tfrac43$, but do not have strongly nonnegative curvature if $\lambda$ is sufficiently close to $\tfrac43$.
\end{remark}

\section{Cohomogeneity one manifolds}
\label{sec:cohom1}

A \emph{cohomogeneity one manifold} is a Riemannian manifold $(M,\g)$ with an isometric action by a compact Lie group $\G$ such that the orbit space $M/\G$ is $1$-dimensional.
It is natural to investigate strongly nonnegative curvature among these manifolds after observing that all compact homogeneous (that is, \emph{cohomogeneity zero}) spaces admit strongly nonnegative curvature, see Section~\ref{sec:basics}. After briefly describing the basic structure of cohomogeneity one manifolds (see, e.g., \cite{mybook,gwz,grove-ziller-annals} for details), we strengthen the gluing construction of Grove and Ziller~\cite{grove-ziller-annals} from $\sec\geq0$ to strongly nonnegative curvature, proving Theorem~\ref{thm:B}.

\subsection{Topological structure}
The orbit space $M/\G$ of a cohomogeneity one manifold $M$ is, up to rescaling, isometric to one of $\R$, $S^1$, $[0,+\infty)$ or $[-1,1]$. In the first two cases, all orbits are principal, and hence the quotient map $q\colon M\to M/\G$ is a fiber bundle. In the last two cases, there are nonprincipal orbits $S$ corresponding to boundary points of $M/\G$, which are called \emph{singular} or \emph{exceptional}, according to their dimension being respectively smaller or equal to that of principal orbits.

If $M/\G=[0,+\infty)$, then $M$ is equivariantly diffeomorphic to the total space of a disk bundle over the unique nonprincipal orbit $S$. More precisely, fix $p\in S$, denote by $\K=\G_p$ its isotropy group, and denote by $V=\nu_p S$ the normal space to $S$. The slice representation $\rho\colon\K\to\O(V)$ is transitive on spheres in $V$. By the Slice Theorem, $M$ is $\G$-equivariantly diffeomorphic to the quotient $\G\times_\K V$ of the product $\G\times V$ by the action of $\K$ on $\G\times V$ given by $k\cdot(g,v)=(gk^{-1},\rho(k)v)$. Fixing a unit vector $v_0\in V$, the principal isotropy group is $\H=\K_{v_0}$, and the unit sphere $S(V)$ in $V$ is $\K$-equivariantly diffeomorphic to $\K/\H$.

If $M/\G=[-1,1]$, then $M$ is $\G$-equivariantly diffeomorphic to the union of two disk bundles as above, one over each of the two nonprincipal orbits $S_\pm=\G/\K_\pm$, glued along their common boundary, which is a principal orbit $\G/\H$.

\subsection{Strongly nonnegative curvature}
The construction of cohomogeneity one metrics with strongly nonnegative curvature is straightforward in case $M/\G$ is one of $\R$, $S^1$, or $[0,+\infty)$. We thus focus on the more involved case $M/\G=[-1,1]$, which requires that nonprincipal orbits $S_\pm$ have codimension $\leq2$. 
We follow the same strategy as in Grove and Ziller~\cite{grove-ziller-annals} to glue two disk bundles. Namely, we construct metrics $\g_\pm$ with strongly nonnegative curvature on each ``half'' $M_\pm=\G\times_{\K_\pm} V_\pm$, which outside of a compact set are isometric to $\G/\H\times[0,\varepsilon)$ with a product metric $\g_0+\dd t^2$, where $\g_0$ is normal homogeneous. This is achieved with a \emph{scale up/scale down} procedure involving the bi-invariant metric on $\G$. Since the construction is the same on each half, we henceforth drop the subscripts $_\pm$.

The desired metric $\g$ on $\G\times_\K V$ is induced by a metric on $\G\times V$ of the form $L+\dd t^2 + f(t)^2\dd\theta^2$, where $L$ is a left-$\G$-invariant and right-$\K$-invariant metric on~$\G$, $f(t)$ is an odd smooth function such that $f'(0)=1$ and $f(t)>0$ for all $t>0$, and $\dd\theta^2$ is the round metric on the unit sphere $S(V)$. Let $\pi\colon\G\times V\to\G \times_{\K} V$ denote the quotient map, and $\h\subset \k\subset\mathfrak g$ the Lie algebras of $\H\subset\K\subset\G$. Write $L$-orthogonal decompositions $\mathfrak g=\k\oplus\mm$ and $\k=\h\oplus\p$. We use subscripts to denote the components in these subspaces, e.g., $X_\k$ and $X_\mm$ are the components of $X\in\mathfrak g$ in $\k$ and $\mm$ respectively.
Routine computations show that the vertical and horizontal spaces of the Riemannian submersion $\pi\colon\G\times V\to \G\times_\K V$ at $(e,tv_0)$ are given by
\begin{equation}\label{eq:vertical-horizontal}
\begin{aligned}
\mathcal{V}&=(\h\times 0)\oplus\big\{\big(\!-\!X,X^*_{tv_0}\big) \ |\ X\in\p\big\},\\
\mathcal{H}&=(\mm\times 0)\oplus\big\{\big(f(t)^2BY,Y^*_{tv_0}\big) : Y\in\p\big\}\oplus\operatorname{span}\big\{\tfrac{\partial}{\partial t}\big\},
\end{aligned}
\end{equation}
where $X^*_{tv_0}=\frac{\dd}{\dd s}\rho(\exp(sX))tv_0\big|_{s=0}$ is the value at $tv_0\in V$ of the action field $X^*$ induced by $X\in\p$, and $B$ is the $L$-symmetric automorphism $B\colon\p\to\p$ such that $L(\cdot,B\cdot)=\dd\theta^2$. 

The following description of the metric on the principal orbits can be obtained from the above splitting, see for instance \cite{cheeger,grove-ziller-annals}.

\begin{lemma}[Scale down]\label{lem:prescribing}
Using the above notation, for each $t>0$, we have:
\begin{enumerate}[\rm (i)]
\item The metric $\langle\cdot,\cdot\rangle$ on the principal orbit $\G\big([e,tv_0]\big)\subset\G\times_\K V$ induced by the metric $L+\dd t^2+f(t)^2\dd\theta^2$ on $\G\times V$ is given by $L(\cdot,C\cdot)$, where $C\colon\mm\oplus\p\to\mm\oplus \p$ is the $L$-symmetric automorphism defined as
\begin{equation*}
C=\diag\big(\id,\,f(t)^2\,B\,(\id + f(t)^2B)^{-1}\big).
\end{equation*}
\item Suppose that $B=b\,\id$ for some $b>0$, and $f(t)^2>\tfrac1b$. Define a metric $L'(\cdot,\cdot)=L(\cdot,D\cdot)$ on $\mathfrak g=\mm\oplus\k$, where
\begin{equation*}
D=\diag\left(\id,\tfrac{f(t)^2b}{f(t)^2b-1}\id\right).
\end{equation*}
Then $L'$ is $\Ad_\K$-invariant, and the metric $L'+ \dd t^2 + f(t)^2\dd\theta^2$ on $\G\times V$ induces the metric $L|_{\mm\oplus\p}$ on the principal orbit $\G\big([e,tv_0]\big)$.
\end{enumerate}
\end{lemma}

The second key ingredient in the construction is the following strengthening of \cite[Prop.\ 2.4]{grove-ziller-annals}, which states that a bi-invariant metric $Q$ on $\G$ retains (strongly) nonnegative curvature when it is dilated by a factor of up to $\tfrac43$ in the direction of an abelian subgroup $\A\subset \G$. This is accomplished (just as in \cite[Prop.\ 2.4]{grove-ziller-annals}) by viewing this process as a ``backwards'' Cheeger deformation, that is, the enlarged metric on $\G$ is induced by a submersion from $\G\times\A$ with a \emph{semi-Riemannian} metric.

\begin{lemma}[Scale up]
\label{lem:fourthirds}
Let $(\G,Q)$ be a Lie group with bi-invariant metric, $\a$ be an abelian subalgebra of $\mathfrak g$, and $\mathfrak n$ be its $Q$-orthogonal complement. The left-invariant metrics $Q_t=t\,Q|_\a\oplus Q|_{\mathfrak n}$ on $\G$ have strongly nonnegative curvature for all $0<t\leq\tfrac43$.
\end{lemma}

\begin{proof}
The result is obvious for $t=1$, since the curvature operator of $(\G,Q)$ is positive-semidefinite, hence $(\G,Q)$ trivially has strongly nonnegative curvature.

Consider $t>0$, $t\neq1$, and let $\A$ be the unique connected Lie subgroup of $\G$ with Lie algebra $\a$. Endow $\G\times\A$ with the semi-Riemannian product metric $Q+\tfrac{t}{1-t}Q|_\a$. A straightforward computation shows that the map
\begin{equation}\label{eq:setupwallach}
\pi\colon\left(\G\times\A,Q+\tfrac{t}{1-t}Q|_\a\right)\to (\G,Q_t), \quad \pi(g,a)=a^{-1}g,
\end{equation}
is a semi-Riemannian submersion. Indeed, the horizontal lift of $X\in\mathfrak g$ is given by
\begin{equation*}
\overline X=\big(X_\mathfrak n+tX_\mathfrak a,(t-1)X_\mathfrak a\big)\in\mathfrak g\oplus\mathfrak a,
\end{equation*}
and $\big(Q+\tfrac{t}{1-t}Q|_\a\big)\big(\overline X,\overline Y\big)=Q_t(X,Y)$ for all $X,Y\in\mathfrak g$.

The $A$-tensor of this semi-Riemannian submersion can be computed as
\begin{equation*}
\begin{aligned}
A_X Y&=\tfrac12 [\overline X,\overline Y]^\mathcal V\\
&=\tfrac12 \big([X_\mathfrak n+tX_\a,Y_\mathfrak n+tY_\a],0\big)^\mathcal V\\
&=\tfrac12 \big((1-t)[X_\mathfrak n,Y_\mathfrak n]_\a,(1-t)[X_\mathfrak n,Y_\mathfrak n]_\a\big).
\end{aligned}
\end{equation*}
Thus,
\begin{equation}\label{eq:alpha43}
\begin{aligned}
\big\langle\alpha(\overline X\wedge\overline  Y),\overline Z\wedge\overline W\big\rangle &= \big\langle A_{\overline{X}} \overline{Y},A_{\overline{Z}} \overline{W}\big\rangle\\
&=\big(1+\tfrac{t}{1-t}\big)Q\big(\tfrac12(1-t)[X_\mathfrak n,Y_\mathfrak n]_\a,\tfrac12(1-t)[Z_\mathfrak n,W_\mathfrak n]_\a\big)\\
&=\tfrac{1-t}{4}Q\big([X_\mathfrak n,Y_\mathfrak n]_\a,[Z_\mathfrak n,W_\mathfrak n]_\a\big).
\end{aligned}
\end{equation}

By the Gray-O'Neill formula \eqref{eq:oneill} and \eqref{eq:RG}, the curvature operator of $(\G,Q_t)$ is
\begin{equation*}
\begin{aligned}
\langle R_t(X\wedge Y),Z\wedge W\rangle_t &=\tfrac14 Q\big([X_\mathfrak n+tX_\a,Y_\mathfrak n+tY_\a],[Z_\mathfrak n+tZ_\a,W_\mathfrak n+tW_\a]\big)\\
&\quad +3\langle\alpha(\overline X\wedge\overline  Y),\overline Z\wedge\overline W\rangle -3\b(\alpha)(\overline X,\overline Y,\overline Z,\overline W),
\end{aligned}
\end{equation*}
where $\b$ is the Bianchi map. Let us expand the above first term, by separating the components in $\a$ and in $\mathfrak n$ and using that $[\mathfrak a,\mathfrak n]\subset\mathfrak n$.
\begin{equation*}
\begin{aligned}
\langle R_t(X\wedge Y),Z\wedge W\rangle_t &=\tfrac14 Q\big([X_\mathfrak n,Y_\mathfrak n]_\a,[Z_\mathfrak n,W_\mathfrak n]_\a\big)\\
&\quad+\tfrac14 Q\big([X_\mathfrak n,Y_\mathfrak n]_\mathfrak n,[Z_\mathfrak n,W_\mathfrak n]_\mathfrak n\big)\\
&\quad+t^2Q([X_\mathfrak n,Y_\a]+[X_\a,Y_\mathfrak n],[Z_\mathfrak n,W_\a]+[Z_\a,W_\mathfrak n])\\
&\quad +3\langle\alpha(\overline X\wedge\overline  Y),\overline Z\wedge\overline W\rangle -3\b(\alpha)(\overline X,\overline Y,\overline Z,\overline W),
\end{aligned}
\end{equation*}
Substituting \eqref{eq:alpha43} in the above and combining with the first term, we conclude that
\begin{equation*}
\begin{aligned}
\langle R_t(X\wedge Y),Z\wedge W\rangle_t &=\tfrac{4-3t}{4} Q\big([X_\mathfrak n,Y_\mathfrak n]_\a,[Z_\mathfrak n,W_\mathfrak n]_\a\big)\\
&\quad+\tfrac14 Q\big([X_\mathfrak n,Y_\mathfrak n]_\mathfrak n,[Z_\mathfrak n,W_\mathfrak n]_\mathfrak n\big)\\
&\quad+t^2Q([X_\mathfrak n,Y_\a]+[X_\a,Y_\mathfrak n],[Z_\mathfrak n,W_\a]+[Z_\a,W_\mathfrak n])\\
&\quad -3\b(\alpha)(\overline X,\overline Y,\overline Z,\overline W).
\end{aligned}
\end{equation*}
Therefore, setting $\omega_t(X,Y,Z,W):=3\b(\alpha)(\overline X,\overline Y,\overline Z,\overline W)$, we have that if $0<t\leq\tfrac43$, then $R_t+\omega_t$ is a sum of positive-semidefinite operators, hence positive-semidefinite. Thus, $(\G,Q_t)$ has strongly nonnegative curvature for all $0<t\leq\tfrac43$.
\end{proof}

We now use Lemmas~\ref{lem:prescribing} and \ref{lem:fourthirds} to prove Theorem \ref{thm:B}, in analogy with the $\sec\geq0$ construction of Grove and Ziller~\cite[Thm.\ 2.6]{grove-ziller-annals}.

\begin{proof}[Proof of Theorem \ref{thm:B}]
The other cases being straightforward, let $M$ be a cohomogeneity one $\G$-manifold with $M/\G=[-1,1]$. Let $S_\pm=\G/\K_\pm$ be the nonprincipal orbits, and consider separately each of the two ``halves" $\G\times_{\K_\pm} \!V_\pm$ of $M$, which are disk bundles over $S_\pm$. Fix a bi-invariant metric $Q$ on $\G$. We will construct a metric $\g$ on each disk bundle $\G\times_\K V$ that has strongly nonnegative curvature, and near the boundary is isometric to $\G/\H\times [0,\varepsilon)$ with a product metric, where $\G/\H$ is endowed with the normal homogeneous metric defined by $Q$. Gluing these two halves together along their common boundary $\G/\H$ yields the desired metric on $M$.

If $S=\G/\K$ is exceptional, i.e., has codimension $1$, then the metric induced on $\G\times_\K V$ by the product metric $Q+\dd t^2$ on $\G\times V$ clearly has the desired properties. Thus, assume that $S=\G/\K$ has codimension $2$, which means that $\dim \p=1$. This implies that $\p$ is an abelian subalgebra of $\mathfrak g$, and the standard metric $\dd\theta^2$ on the circle $S^1$ is given by $Q(\cdot,B\cdot)$ where $B=b\,\id$ for some $b>0$, cf.\ Lemma \ref{lem:prescribing} (ii).
Let $f(t)$ be an odd smooth function such that $f'(0)=1$, $f(t)>0$ and $f''(t)\leq0$ for all $t>0$, and $f(t)\equiv a$ is constant for $t\geq t_0$, where $a$ satisfies $a\geq\tfrac{2}{\sqrt{b}}$ so that $\tfrac{a^2b}{a^2b-1}\leq\tfrac43$.
The \emph{cigar metric} $\dd t^2+f(t)^2\dd\theta^2$ on $V$ has positive-semidefinite curvature operator, hence trivially has strongly nonnegative curvature.
Consider the \emph{scaled up} metric $L'(\cdot,\cdot)=Q(\cdot,E\cdot)$ on $\mathfrak g=\mm\oplus\k$, where $E\colon\mm\oplus\p\oplus\h\to\mm\oplus\p\oplus\h$
is given by
\begin{equation}\label{eq:E}
E=\diag\left(\id,\tfrac{a^2b}{a^2b-1}\id,\id\right).
\end{equation}
Since this metric $L'$ on $\G$ has strongly nonnegative curvature by Lemma~\ref{lem:fourthirds}, the product metric $L'+\dd t^2+f(t)^2\dd \theta^2$ on $\G\times V$ also has strongly nonnegative curvature.
It is easy to see that $L'$ is $\Ad_\K$-invariant \cite[p.\ 341]{grove-ziller-annals}, and as $\K$ acts orthogonally on $V$, we have that $L'+\dd t^2+f(t)^2\dd \theta^2$ descends to a \emph{scaled down} metric $\g$ on $\G\times_\K V$. The quotient map $\pi\colon\G\times V\to\G\times_\K V$ is hence a Riemannian submersion, so $(\G\times_\K V,\g)$ has strongly nonnegative curvature.
Finally, Lemma~\ref{lem:prescribing} (ii) implies that, for any $t\geq t_0$, the metric induced by $\g$ on the principal orbit $\G\big([e,tv_0]\big)$ is the normal homogeneous metric defined by $Q$. This concludes the construction of the desired metric with strongly nonnegative curvature on each half of $M$.
\end{proof}

\begin{remark}\label{rem:colagens}
Instead of gluing the two halves $(\G\times_{\K_\pm} \!V_\pm,\g_\pm)$ of $M$ identifying their common boundary $\G/\H$ via the identity map, one may use any other isometry  $\phi$ of $\G/\H$. Despite being the union of the same two cohomogeneity one disk bundles, the resulting manifold $M'=\G\times_{\K_-} V_-\cup_\phi \G\times_{\K_+} V_+$ is in general not diffeomorphic to $M$, and unless $\phi\in \N(\H)/\H$, it does not have a global isometric $\G$-action, but has strongly nonnegative curvature. Obviously, one may also replace one of the disk bundles $\G\times_{\K_+}\!V_+$ by any other disk bundle with the same boundary; e.g., gluing two copies of the same disk bundle $\G\times_{\K_-}\!V_-$ one produces the double of that bundle.

For instance, consider the cohomegeneity one $2$-disk bundle determined by the groups $\H=\U(n-1)$, $\K=\U(n-1)\U(1)$, and $\G=\U(n)$. This is the normal disk bundle of $\C P^{n-1}\subset\C P^{n}$, and is diffeomorphic to the complement of a disk in $\C P^n$, see Section~\ref{sec:cheegermflds} for details. Gluing with the identity map on $\G/\H=S^{2n-1}$, the result is $\C P^n\#\overline{\C P}^n$, while with the antipodal map it is $\C P^n\#\C P^n$. The former manifolds admit a cohomogeneity one $\G$-action, however the latter do not for $n=2,3$ \cite{hoelscher,parker}.
\end{remark}

\subsection{Principal and associated bundles}
The class of manifolds that can be shown to admit metrics with strongly nonnegative curvature due to Theorem~\ref{thm:B} extends far beyond that of cohomogeneity one manifolds, thanks to the associated bundle construction. Recall that given a principal $\G$-bundle $P$ and an isometric $\G$-action on a manifold~$F$, the associated bundle $P\times_\G F$ is the orbit space of a free $\G$-action on $P\times F$, see \cite[\S 3.2]{mybook}. Strongly nonnegative curvature, just as $\sec\geq0$, is preserved under products and Riemannian submersions~\cite{strongpos}; so if both $P$ and $F$ have strongly nonnegative curvature, then so does $P\times_{\G} F$. In the remainder of this section, we list all currently known applications of this technique. 

As an important example, all principal $\SO(k)$-bundles over $S^4$ have a cohomogeneity one $\SO(3)\times\SO(k)$-action  with singular orbits of codimension $2$, see \cite[Thm.\ F]{grove-ziller-annals} and \cite[Thm.\ 2.10]{bible}, and hence metrics with strongly nonnegative curvature by Theorem~\ref{thm:B}. Via the associated bundle construction, it follows that all vector bundles and sphere bundles over $S^4$ have complete metrics with strongly nonnegative curvature. This accounts for $20$ of the $28$ oriented diffeomorphism types of spheres in dimension $7$, which includes all \emph{Milnor exotic spheres}. It has been announced that the $8$ remaining exotic $7$-spheres are orbit spaces of free $\Sp(1)$-actions on cohomogeneity one manifolds of dimension $10$ with codimension $2$ singular orbits~\cite{exotic7spheres}, hence they also admit metrics with strongly nonnegative curvature by Theorem~\ref{thm:B}.

Using these techniques on other principal $\G$-bundles, one obtains the following comprehensive list of instances where the ``converse" to the Soul Theorem of Cheeger and Gromoll~\cite{cheeger-gromoll} explained in the Introduction is currently known to hold:
\begin{corollary}\label{cor:C}
The total space of the following vector bundles, and the corresponding sphere bundles, admit complete metrics with strongly nonnegative curvature:
\begin{enumerate}[\rm (i)]
\item All vector bundles over $S^4$ and $S^5$;
\item All vector bundles over $S^7$ of rank $3$, and $88$ of the $144$ of rank $4$;
\item All vector bundles over $\C P^2$ with nontrivial second Stiefel-Whitney class;
\item All complex rank $2$ vector bundles over $\C P^2$ whose first Chern class $c_1$ is odd, or whose $c_1$ is even and the discriminant $\Delta:=c_1^2-4c_2$ satisfies $\Delta\equiv 0\mod 8$;
\item All vector bundles of rank $\geq6$ over $\C P^2$, $S^2\times S^2$ and $\C P^2\#\overline{\C P}^2$;
\item A representative of any class of stable vector bundles over any compact rank one symmetric space.
\end{enumerate}
\end{corollary}

Details on how to construct a homogeneous or cohomogeneity one structure with singular orbits of codimension $\leq2$ on the corresponding principal $\G$-bundles can be found in Grove and Ziller~\cite[Thm.~B, Prop.~3.14]{grove-ziller-annals} and Rigas~\cite{rigas} for (i), Grove and Ziller~\cite[Cor.\ 3.13]{grove-ziller-annals} for (ii), Grove and Ziller~\cite[Thm.\ 1, Thm.\ 2, Cor.]{grove-ziller-tams} for (iii), (iv), and (v), respectively, and Rigas~\cite{rigas} and Gonz\'alez-\'Alvaro~\cite{david} for (vi).

\section{Connected sum of two Compact Rank One Symmetric Spaces}\label{sec:cheegermflds}

One of the main inspirations for the cohomogeneity one gluing construction of Grove and Ziller~\cite{grove-ziller-annals} described in the previous section was an earlier result of Cheeger~\cite{cheeger} about gluing two compact rank one symmetric spaces (CROSS). In this section, we also strengthen this construction from $\sec\geq0$ to strongly nonnegative curvature, proving Theorem~\ref{thm:A}.

We follow the same strategy as in Cheeger~\cite{cheeger}, showing that the complement of a ball in each CROSS admits a metric with strongly nonnegative curvature, which near the boundary is isometric to the round cylinder $S^{d-1}\times [0,\varepsilon)$. In this way, any two such objects of the same dimension $d$ can be glued together along their boundary $S^{d-1}$, with an identification that preserves or reverses the orientation.

\subsection{Geometric structure}
There is a natural cohomogeneity one $\G$-action \emph{with a fixed point} $S_-=\{p\}$ on each CROSS, see e.g.\ \cite[\S 6.3]{mybook} for details. Using the notation from Section~\ref{sec:cohom1}, the groups in these actions are given in Table~\ref{tab:groups}.
\begin{table}[htf]
\caption{Cohomogeneity one actions with a fixed point in a CROSS}\label{tab:groups}
\begin{tabular}{lllllll}
 & $\G$ & $\K_-$ & $\K_+$ & $\H$ & $V_-$ & $V_+$\\
\noalign{\smallskip}\hline\noalign{\smallskip}
$S^n$ &$\SO(n)$ &$\SO(n)$ & $\SO(n)$  & $\SO(n-1)$ &$\R^n$ & $\R^n$\\
$\R P^n$ &$\SO(n)$ & $\SO(n)$ & $\mathrm{S}(\O(n-1)\O(1))$ & $\SO(n-1)$ & $\R^n$ & $\R$\\
$\C P^n$ &$\U(n)$ &$\U(n)$ & $\U(n-1)\U(1)$  & $\U(n-1)$ & $\C^n$& $\C\cong\R^2$\\
$\Hr P^n$ &$\Sp(n)$ &$\Sp(n)$  & $\Sp(n-1)\Sp(1)$  & $\Sp(n-1)$ & $\Hr^n$ & $\Hr\cong\R^4$\\
$\Ca P^2$ &$\Spin(9)$&$\Spin(9)$  & $\Spin(8)$ & $\Spin(7)$ & $\Ca^2$ & $\Ca\cong\R^8$\\
\noalign{\smallskip}\hline\noalign{\smallskip}
\end{tabular}
\end{table}

All inclusions above are matrix block embeddings, except for $\Spin(8)\subset\Spin(9)$ which comes from the spin representation. We may restrict our attention to the last $3$ cases, since $S^n$ and $\R P^n$ clearly have metrics with these desired properties.\footnote{In addition, these cases can be ignored since $M\# S^n\cong M$, and $M\# \R P^n$ is double-covered by $M\#\overline M$, where $\overline M$ denotes $M$ with the reversed orientation.}

Denote the above projective spaces by $\KK P^n$, where $\KK$ is one of the real normed division algebras $\R$, $\C$, $\Hr$, or $\Ca$, and set $k=\dim_\R\KK$. The principal orbits $\G/\H$ are Berger spheres $S^{kn-1}$, which are boundaries of metric balls centered at~$p$. The other singular orbit $S_+=\G/\K_+$ is a totally geodesic $\KK P^{n-1}$, which is the cut locus of~$p$, and the homogeneous bundles $\K_+/\H\to \G/\H\to\G/\K_+$ are Hopf bundles $S^{k-1}\to S^{kn-1}\to\KK P^{n-1}$. Thus, the complement $M$ of a metric ball centered at $p$ is diffeomorphic to the normal bundle of $\operatorname{Cut}(p)=\KK P^{n-1}$. In particular, this ``half" $M\cong\G\times_{\K_+}\! V_+$ of the cohomogeneity one manifold is a disk bundle exactly as those in the previous section. However, note that the codimension of $S_+$ in $\KK P^n$ is $k$, so the gluing method in the proof of Theorem~\ref{thm:B} does not apply unless $\KK=\C$, cf.\ Remark~\ref{rem:colagens}. In addition, we remark that the normal homogeneous metric on $\G/\H=S^{kn-1}$ is \emph{not isometric} to the round metric unless $kn=2$ or $4$, so a different construction is required.\footnote{Note that if $d>4$ and a metric on $S^{d-1}$ is invariant under more than one transitive $\G$-action with $\G$ among $\U(n)$, $\Sp(n)$ and $\Spin(9)$, then this metric is round.}

\subsection{Strongly nonnegative curvature}
In order to carry out the above mentioned strategy, we need to construct a metric on the disk bundle $M\cong\G\times_{\K}\! V$ (we drop the subscript $_+$ to simplify notation) with strongly nonnegative curvature, which is isometric to a round cylinder near the boundary $\G/\H$. The cases of $\C P^n$ and $\Hr P^n$ can be easily dealt with because $\K=\H\times\mathsf{L}$, where $\mathsf{L}$ is respectively $\U(1)$ and $\Sp(1)$, and the slice representation $\rho\colon\K\to\O(V)$ factors through $\mathsf{L}$. This implies that $M\cong\G\times_\K V\cong(\G/\H)\times_\mathsf{L} V$, so that a metric on $M$ with (strongly) nonnegative curvature may be produced from $\mathsf{L}$-invariant metrics with (strongly) nonnegative curvature on the sphere $\G/\H$ and the vector space $V$. This is precisely the argument found in Cheeger~\cite{cheeger}.
On the other hand, the case of the Cayley plane $\Ca P^2$ does not admit such a simplification, and a more delicate (but ultimately analogous) argument is required. For the sake of completeness, we state it below in greater generality than what is needed for proving Theorem \ref{thm:A}, using the same notation as in Lemma \ref{lem:prescribing}.

\begin{lemma}\label{lem:cheeger}
Let $\G\times_\K V$ be the normal disk bundle of a singular orbit in a cohomogeneity one manifold. Let $f(t)$ be an odd smooth function such that $f'(0)=1$, $f(t)>0$ and $f''(t)\leq0$ for all $t>0$. If the metric on $\G/\H$ induced by $L$ has strongly nonnegative curvature (respectively, $\sec\geq0$), then the metric on $\G\times_\K V$ induced by $L+\dd t^2 + f(t)^2\dd\theta^2$ has strongly nonnegative curvature (respectively, $\sec\geq0$).
\end{lemma}

\begin{proof}
We restrict ourselves to the claims regarding strongly nonnegative curvature, as the case of $\sec\geq0$ is similar and less involved.

Since $\G$ acts on $\G\times_\K V$ with cohomogeneity one, it is enough to consider points along a radial geodesic $\gamma(t)=\pi(e,tv_0)$, $t\geq0$. Moreover, we may assume $t>0$, since strongly nonnegative curvature is a closed condition.
Applying the Gray-O'Neill formula \eqref{eq:oneill} to the Riemannian submersion $\pi\colon \G\times V\to \G\times_\K V$, we have
\begin{equation*}
R^{\G\times_\K V}=R^{\G\times V}|_{\wedge^2\mathcal{H}}+3\alpha -3\b(\alpha),
\end{equation*}
where $\alpha=A^*A$, and $A\colon\wedge^2\mathcal{H}\to\mathcal{V}$ is the Gray-O'Neill tensor $A(X\wedge Y)=\tfrac12[X,Y]^\mathcal{V}$.
We proceed by estimating from below the first 2 terms in the right hand side of the above formula.

Regarding the first term $R^{\G\times V}$, since $\G\times V$ is endowed with the product metric $L+\dd t^2 + f(t)^2\dd\theta^2$ and $f(t)$ is concave, we have that for all $\beta\in\wedge^2\mathcal{H}\subset \wedge^2(\mathfrak g\times V)$,
\begin{equation}\label{eq:est1}
\left<R^{\G\times V}\beta,\beta\right>=  \left< R^\G \beta_1,\beta_1\right> +\left<R^V\beta_2,\beta_2\right> \geq \left< R^\G \beta_1,\beta_1\right>,
\end{equation}
where  $\beta=\beta_1+\beta_2+\beta_3$  has components $\beta_1\in\wedge^2\mathfrak g$, $\beta_2\in \wedge^2V$, and $\beta_3\in\mathfrak g\otimes V$. Note that by \eqref{eq:vertical-horizontal}, we have $\beta_1\in\wedge^2(\mm\oplus\p)$.

To estimate the second term $3\alpha$, we write $A=A_1+A_2$ according to the splitting \eqref{eq:vertical-horizontal} of the vertical space $\mathcal V$. Namely,  $A_1$ is the component of $A$ with image in $\mathfrak{h}\times \{0\}$ and $A_2$ in $\big\{\big(\!-\! X,X^*_{tv_0}\big) : X\in\p\big\}$. It then follows that
\begin{equation}\label{eq:est2}
\left< \alpha (\beta), \beta\right>= \left<A_1\beta,A_1\beta\right> + \left<A_2\beta,A_2\beta\right> \geq  \left<A_1\beta,A_1\beta\right>=\left<\alpha_{\G/\H}\beta_1,\beta_1\right>,
\end{equation}
where $\alpha_{\G/\H}=A_{\G/\H}^*A_{\G/\H}$, and  $A_{\G/\H}\colon\wedge^2(\mathfrak{m}\oplus\mathfrak{p})\to\mathfrak{h}$ is the $A$-tensor of the Riemannian submersion $\G\to \G/\H$. The last equality in \eqref{eq:est2} follows from writing $\beta=\sum X^i\wedge Y^i$, for $X^i,Y^i\in \mathcal H$, where $X^i=(X^i_1,X^i_2), Y^i=(Y^i_1,Y^i_2) \in\mathfrak g\oplus V$, and computing
\begin{equation*}
\begin{aligned}
\langle A_1\beta,A_1\beta\rangle&=\left\langle\sum_i \tfrac12 [X^i,Y^i]_{\mathfrak h\times\{0\}},\sum_j \tfrac12 [X^j,Y^j]_{\mathfrak h\times\{0\}}\right\rangle\\
&=\left\langle\sum_i \tfrac12 [X^i_1,Y^i_1]_{\mathfrak h},\sum_j \tfrac12 [X^j_1,Y^j_1]_{\mathfrak h}\right\rangle\\
&= \alpha_{\G/\H}(\beta_1,\beta_1).
\end{aligned}
\end{equation*}

Since $\G/\H$ has strongly nonnegative curvature, there is a $4$-form $\omega_{\G/\H}\in\wedge^4(\mm\oplus\p)$ such that $R^{\G/\H}+\omega_{\G/\H}$ is positive-semidefinite. Define $\omega\in\wedge^4\mathcal{H}$ by
\begin{equation*}
\omega(\beta)=\omega_{\G/\H}(\beta_1)+3\b(\alpha)(\beta)-3\b(\alpha_{\G/\H})(\beta_1).
\end{equation*}
Combining the estimates \eqref{eq:est1}, \eqref{eq:est2}, and the Gray-O'Neill formula \eqref{eq:oneill} applied to the Riemannian submersions $\G\to \G/\H$ and $\G\times V\to \G\times_\K V$, we have
\begin{equation*}
\begin{aligned}
\left<\big(R^{\G\times_\K V} +\omega\big)\beta,\beta\right>
&= \left<R^{\G\times V}(\beta)+3\alpha(\beta)+\omega_{\G/\H}(\beta_1)-3\b(\alpha_{\G/\H})(\beta_1),\beta\right>\\
&\geq \left<R^{\G}(\beta_1)+3\alpha_{\G/\H}(\beta_1)-3\b(\alpha_{\G/\H})(\beta_1) +\omega_{\G/\H}(\beta_1), \beta_1 \right> \\
&\geq \left<\big(R^{\G/\H} +\omega_{\G/\H}\big)\beta_1, \beta_1 \right> \geq 0
\end{aligned}
\end{equation*}
concluding the proof that $\G\times_\K V$ has strongly nonnegative curvature.
\end{proof}

We are finally ready to give a proof of Theorem \ref{thm:A}, using Lemmas \ref{lem:prescribing} and \ref{lem:cheeger}.

\begin{proof}[Proof of Theorem \ref{thm:A}]
As discussed above, it suffices to construct a metric with strongly nonnegative curvature on the complement of a ball in $\C P^n$, $\Hr P^n$, and $\Ca P^2$, which is isometric to a round cylinder near the boundary. Each of these manifolds is diffeomorphic to a disk bundle $\G\times_\K V$, where $\G$, $\K$ and $V$ are given in Table~\ref{tab:groups}. In all cases, $\p$ is $\Ad_\H$-irreducible, hence the assumption that $B=b\,\id$ in Lemma \ref{lem:prescribing} (ii) is satisfied due to Schur's Lemma. Assume that $f(t)$ is a function as in Lemma~\ref{lem:cheeger} and $f(t)\equiv a$ is constant for $t\geq t_0$, with $a^2>\tfrac1b$. Let $L$ be the left-$\G$-invariant and right-$\K$-invariant metric on $\mathfrak g$ which induces the round metric on the sphere $\G/\H$, and consider the \emph{scaled up} metric $L'(\cdot,\cdot)=L(\cdot,E\cdot)$ on $\mathfrak g$, where $E$ is given by \eqref{eq:E}. Since $L'$ converges to $L$ as $a\to\infty$, and $L$ induces a metric with strongly positive curvature on $\G/\H$ (which is an open condition), it follows that the constant $a$ can be chosen sufficiently large so that $L'$ also induces a metric with strongly positive curvature on $\G/\H$. Therefore, the metric on $\G\times_\K V$ induced by $L'+\dd t^2+f(t)^2\dd \theta^2$ has strongly nonnegative curvature, by Lemma \ref{lem:cheeger}. Finally, according to Lemma~\ref{lem:prescribing} (ii), the disk bundle $\G\times_\K V$ with this metric is isometric to a round cylinder near the boundary, concluding the proof.
\end{proof}

\begin{remark}\label{rem:totaro}
Some connected sums of two CROSS, such as $\C P^n\#\overline{\C P}^n$, $\Hr P^n\#\overline{\Hr P}^n$, and $\Ca P^2\#\overline{\Ca P}^2$ are diffeomorphic to biquotients, providing an alternative way of endowing them with metrics of strongly nonnegative curvature. Nevertheless, there are also some connected sums of CROSS, such as $\C P^8\#\Ca P^2$, $\Hr P^4\#\Ca P^2$, and $\Ca P^2\#\Ca P^2$, that are not even homotopy equivalent to biquotients \cite[Thm.\ 2.1]{totaro-biquotients}.
\end{remark}

\begin{remark}
By the proof of Theorem~\ref{thm:A}, it suffices that $\G/\H$ has strongly positive curvature and $\mathfrak p$ is irreducible, for the disk bundle $\G\times_{\K} V$ to have a metric with strongly nonnegative curvature which is product near the boundary. Thus, one is led to asking what manifolds can be obtained by gluing two such disk bundles along their common boundary. Combining the classifications of homogeneous spaces $\G/\H$ with strongly positive curvature~\cite{strongpos,moduli-flags} and homogeneous structures $\K/\H$ on spheres, it follows that the only possibilities are the above connected sums of two CROSS, besides doubles, homogeneous spaces, and biquotients. Thus, despite the relatively general framework provided above, this method cannot produce any new examples.
\end{remark}

\section{Open manifolds and their souls}\label{sec:guijarro}

The celebrated Soul Theorem of Cheeger and Gromoll~\cite{cheeger-gromoll} states that a complete open manifold $M$ with $\sec\geq0$ is diffeomorphic to the normal bundle $\nu S$ of a totally convex (hence totally geodesic) compact submanifold $S$ without boundary, called a \emph{soul} of $M$. Note that if $M$ has strongly nonnegative curvature, then its soul $S$ also has strongly nonnegative curvature~\cite[Prop.\ 2.6]{strongpos}.

Theorem \ref{thm:D} in the Introduction is a direct consequence of the following, which is the analogue of a result of Guijarro~\cite[Thm.\ A]{guijarro} for strongly nonnegative curvature.

\begin{theorem}\label{thm:guijarroplus}
Let $(M,\g)$ be a compact convex manifold with smooth boundary, strongly nonnegative curvature, and soul $S$. There exists a complete metric $\g'$ on the interior of $M$ with strongly nonnegative curvature, such that $S$ remains a soul, and $(M,\g')$ is isometric to the product $\nu_1(S)\times [0,+\infty)$ outside a compact neighborhood of $S$.
\end{theorem}

\begin{proof}
It follows from Perelman \cite{perelman-soul} that there exists $r_*>0$, smaller than the focal radius of the soul $S$, such that the tubular neighborhood $D_r$ of radius $r$ around $S$ is convex for all $0\leq r\leq r_*$ and the Sharafutdinov retraction $\mathrm{sh}\colon D_{r_*}\to S$ is $C^\infty$, see Guijarro~\cite[Lemma 2.2]{guijarro} and Guijarro and Walschap~\cite[Prop.\ 2.4]{guijarro-walschap}. Using this fact, Guijarro~\cite{guijarro} constructed a smooth convex hypersurface $N\subset D_{r_*}\times\R$, given by the union of the graph of a function $\Gamma\colon D_{r_1}\to\R$ that vanishes identically on $D_{r_0}$, and the cylinder $\partial D_{r_1}\times[1,+\infty)$, for some $0<r_0<r_1<r_*$. In particular, $N$ is diffeomorphic to the normal bundle $\nu S$, and hence to $M$. Since $(M\times\R,\g+\dd t^2)$ has strongly nonnegative curvature and $N\subset M\times\R$ is convex, the induced metric on $N$ also has strongly nonnegative curvature by the Gauss equation. The desired metric $\g'$ is obtained pulling back this induced metric by the diffeomorphism $N\cong M$.
\end{proof}

A consequence of Theorem~\ref{thm:guijarroplus} is the existence of a metric with strongly nonnegative curvature on the double of the normal disk bundle to the soul $S$ of any open manifold with strongly nonnegative curvature, cf.\ Guijarro~\cite[Thm.~1.2]{guijarro}.

\end{document}